\documentclass[twosided]{amsart}
\usepackage{amsmath}
\usepackage{amsfonts}
\usepackage{amssymb,enumerate}
\usepackage{amsthm}
\usepackage{hyperref}
\usepackage{centernot}
\usepackage{stmaryrd}
\usepackage{relsize}
\usepackage{setspace}

\DeclareMathOperator\OO{\mathcal{O}}
\DeclareMathOperator\q{\mathfrak{q}}
\DeclareMathOperator\p{\mathfrak{p}}

\DeclareMathOperator\Spec{\textrm{Spec}}

\theoremstyle{plain}
\newtheorem{theorem}{Theorem}[section]
\newtheorem{lemma}[theorem]{Lemma}

\newtheorem{question}[theorem]{Question}

\newtheorem{notation}[theorem]{Notation}
\begin{document}

\title{A note on a sheaf of real regular functions}
\author{Shahram Mohsenipour}

\address{School of Mathematics, Institute for Research in Fundamental Sciences (IPM)
        P. O. Box 19395-5746, Tehran, Iran}\email{sh.mohsenipour@gmail.com}


\begin{abstract}
Coste and Roy in 1979 defined a structural sheaf on the real Zariski spectrum of a semi-real ring $A$ and asked whether the ring of the global sections is $\sum^{-1}_1 A$ where $\sum_1$ is the multiplicative subset $\{1+\sum_{i=1}^n a_i^2|a_i\in A, n\in\mathbb{N}\}$ of $A$. We give a positive answer to this question for real rings and integral domains.
\end{abstract}
\maketitle
\bibliographystyle{amsplain}

\section{Introduction}
By ring we mean a commutative with identity. The \emph{real Zariski spectrum} of $A$ denoted by $\Spec_R A$ is a topological space whose underlying set is the set of all real prime ideals of $A$ and its topology is described as follows. Let $I$ be an ideal of $A$, we define $V(I)$ as the set all real prime ideals $\mathfrak{p}$ of $A$ with $\mathfrak{p}\supseteq I$. We have
\begin{lemma}
(i) If $I$ and $J$ are two ideals of $A$, then $V(IJ)=V(I)\cup V(J)$.

(ii) If $\{I_{i}\}$ is any set of ideals of $A$, then $V(\sum I_i)=\bigcap V(I_i)$.

(iii) If $I$ and $J$ are two ideals, $V(I)\subseteq V(J)$ if and only if $\sqrt[R]{I}\supseteq\sqrt[R]{J}$.

\end{lemma}
\begin{proof}
The same as (\cite{hartshorne}, Lemma 2.1) and see Theorem \ref{realnulls} below for the real radical $\sqrt[R]{I}$.
\end{proof}
Note that the ring $A$ is semi-real iff $\Spec_R A\neq\emptyset$. Now take the subsets of the form $V(I)$ to be the closed subsets. This defines a topology on $\Spec_R A$ which we call the \emph{real Zariski topology} of $\Spec_R A$. Note that $V(A)=\emptyset$ and $V(\{0\})=\Spec_R A$. For any element $f\in A$, we denote by $D(f)$ the open complement of $V((f))$. It is easy to see that the open sets of the form $D(f)$ form a basis for the topology of $\Spec_R A$. An open set of the form $D(f)$ is called a principal open subset while its complement $V((f))$ is called a principal closed subset.

Now for a ring $A$, Coste and Roy defined a sheaf of rings $\mathcal{O}_{\Spec_R A}$ on the real Zariski spectrum $\Spec_R A$ as follows \cite{coste-roy}. For an open $U\subseteq\Spec_R A$, $\mathcal{O}_{\Spec_R A}(U)$ consists of all functions $s\colon U\rightarrow\coprod_{\mathfrak{p}\in\Spec_R A}\!A_{\mathfrak{p}}$, satisfying the following two conditions. Firstly for all $\mathfrak{p}\in U$ we have $s(\mathfrak{p})\in A_{\p}$, secondly for each $\p\in U$ there is a neighborhood $U_{\p}$ of $\p$ in $U$ and also $a,f\in A$ such that for all $\mathfrak{q}\in U_{\p}$ we have $f\notin\mathfrak{q}$ and $s(\q)=\frac{a}{f}\in A_{\q}$. For $V\subseteq U$ we define $\mathcal{O}_{\Spec_R A}(U)\rightarrow\mathcal{O}_{\Spec_R A}(V)$ by the restriction of functions. Let $X=\Spec_R A$ and $\p\in X$, it is easily seen that there is a canonical isomorphism $\mathcal{O}_{X,\p}\cong A_{\p}$. Though Coste and Roy had motivated from topos theoretic considerations, the sheaf $\mathcal{O}_{\Spec_R A}$ can be regarded as a generalization of affine irreducible real algebraic varieties in the sense of \cite{BCR} in the same vein that affine schemes generalize affine algebraic varieties.

We denote by $\sum_1$ the multiplicative subset $\{ 1+\sum_{i=1}^{n}a_i^{2}\,|a_i\in A, n\in\mathbb{N}\}$ of $A$. Also for a nonzero $f\in A$ we define the following multiplicative subset of $A$.
\[
\sum\nolimits_{f}=\big\{f^{2m}+\sum_{i=1}^{n}a_i^{2}\,|a_i\in A, m\in\mathbb{N}\cup\{0\},n\in\mathbb{N}\big\}.
\]
Obviously if $A$ is a semi-real ring, then $\sum_1\neq\emptyset$ and if $A$ a real ring, then $\sum_f\neq\emptyset$. Now we are able to state Coste and Roy's question.

\begin{question}[\cite{coste-roy}, Page 48]
Do we have $\OO_X(D(f))\cong\sum_f^{-1} A$?
\end{question}

We show in Theorem \ref{main} that the canonical homomorphism from $\sum_f^{-1} A$ into $\OO_X(D(f))$ is always injective and in the case of $A$ being an real ring or an integral domain is surjective. Hence an isomorphism. We don't know whether this is the case for every semi-real ring. We shall make several use of the abstract version of the real Nullstellensatz.

\begin{theorem}[\cite{BCR}, Proposition 4.1.7]\label{realnulls}
Let $A$ be a commutative ring and $I$ an ideal of $A$. Then
\[
\sqrt[R]{I}=\{a\in A\, |\, \exists m\in\mathbb{N}\,\, \exists b_1,\dots,b_p\in A\,\, a^{2m}+b_1^2+\cdots+b_p^2\in I\}
\]
is the smallest real ideal of $A$ containing $I$. The ideal $\sqrt[R]{I}$, called  \textsf{the real radical} of $I$, is the intersection of all real prime ideals containing $I$ (or is $A$ itself if there is no real prime ideal containing $I$).
\end{theorem}

\begin{notation}
\emph{For the sake simplicity we often denote a sum of squares by $\sum x^{2}$,$\sum y^{2}$ and the like. When we say that $\sum x^{2}$ is a sum of square in $A$, we mean that there are $n\in\mathbb{N}$, $a_{1},\dots,a_{n}$ in $A$ such that $\sum x^{2}=\sum_{i=1}^{n}a_i^{2}$. In order to distinguish sums of squares we sometimes use superscripts, e.g., $\sum^{1}x^{2},\sum^{2}x^{2},\dots,\sum^{i}x^{2}$. Also when there is no danger of confusion we use expressions such as $f^{2}\sum x^{2}=\sum (xf)^{2}$.}
\end{notation}

\section{The ring of global sections}

\begin{lemma}
For any $f\in A$, $D(f)$ is quasi-compact.
\end{lemma}
\begin{proof}
Let $D(f_i), i\in I$ be a open cover of $D(f)$, then $V((f))=\bigcap_{i\in I}V((f_i))=V(\sum_{i\in I}(f_i))$. So $f\in\sqrt[R]{\sum_{i\in I}(f_i)}$ and it follows that there are $a_1,\dots,a_k$ in $A$ and $f_{i_1},\dots,f_{i_k}$ such that $a_1f_{i_1}+\cdots+a_kf_{i_k}=f^{2m}+\sum x^2$ for some $m\in\mathbb{N}$. Therefore
\[
D(f)\subseteq D(f^{2m}+\sum x^2)=D(a_1f_{i_1}+\cdots+a_kf_{i_k})\subseteq D(f_{i_1})\cup\cdots\cup D(f_{i_k}),
\]
thus $D(f)$ has a finite subcover.
\end{proof}
Now we will need the following useful lemma.
\begin{lemma}[{\bf Basic Lemma}]\label{basic}
Let $X=\Spec_R A$, $f\in A$ and $s\in\OO_X(D(f))$, then there are $g_1,\dots,g_n\in A$ such that $D(f)=D(g_1)\cup\cdots\cup D(g_n)$ and $s\upharpoonright D(g_i)=\frac{a_i}{g_i}$ for some $a_i\in A$.
\end{lemma}
\begin{proof}
Let $D(f)=\{\p_i|i\in I\}$, then each $\p_i$ has a neighborhood $D(h_i)\subseteq D(f)$ such that $s\upharpoonright D(h_i)=\frac{b_i}{f_i}$ so that for all $\q\in D(h_i)$ we have $f_i\notin\q$. This means that $D(h_i)\subseteq D(f_i)$ for $i\in I$ which implies that $h_i^{2n_i}+\sum^{i}x^2=u_i f_i$ for some $n_i \in \mathbb{N}$ and $\sum^i x^2, u_i\in A$. Thus $s\upharpoonright D(h_i)=\frac{b_i}{f_i}=\frac{u_ib_i}{h_i^{2n_i}+\sum^{i}x^2}$. Now set $a_i=u_ib_ih_i^2$ and $g_i=h_i^2(h_i^{2n_i}+\sum^{i}x^2)$, so $s\upharpoonright D(h_i)=\frac{a_i}{g_i}$. Observing that
\[
D(g_i)=D(h_i^2)\cap D(h_i^{2n_i}+\sum\nolimits^{i}x^2)=D(h_i),
\]
we get $s\upharpoonright D(g_i)=\frac{a_i}{g_i}$. Also $\bigcup_{i\in I}D(g_i)=\bigcup_{i\in I}D(h_i)=D(f)$, then by quasi-compactness we can choose finitely many of $D(g_i)$ covering $D(f)$.
\end{proof}
\begin{theorem}\label{main}
Let $A$ be an real ring or an integral domain, $f\in A$ and let $X=\Spec_R A$, then $\OO_X(D(f))\cong\sum_{f}^{-1} A$. So taking $f=1$ we get
\[
\OO_X(D(1))=\OO_{X}(\Spec_R A)\cong\sum\nolimits_{1}^{-1} A.
\]
\end{theorem}
\begin{proof}
We suppose that $A$ is an real ring. The case of $A$ being an integral domain is dealt with similarly or even simpler because of the fact that all the localizations are the subrings of the quotient filed of $A$. We define a map $\psi\colon\sum_{f}^{-1} A\rightarrow\OO_X(D(f))$ by sending $\frac{a}{f^{2n}+\sum x^2}$ to $s\in\OO_X(D(f))$ where $s(\p)=\frac{a}{f^{2n}+\sum x^2}$ for all $\p\in D(f)$. Obviously $\psi$ is homomorphism. We first show that $\psi$ is injective. {\em Note that for injectivity we do not need $A$ is a real ring or integral domain.} Let
$\psi(\frac{a}{f^{2n}+\sum x^2})=\psi(\frac{b}{f^{2m}+\sum y^2})=s\in\OO_X(D(f))$, then for any $\p\in D(f)$, both $\frac{a}{f^{2n}+\sum x^2}$ and $\frac{b}{f^{2m}+\sum y^2}$ are equal to $s(\p)\in A_{\p}$. Hence there is $h\notin\p$ such that
\[
h\big{[}a(f^{2m}+\sum y^2)-b(f^{2n}+\sum x^2)\big{]}=0.
\]
Let $I$ be the annihilator of $a(f^{2m}+\sum y^2)-b(f^{2n}+\sum x^2)$, then $h\in I$, $h\notin\p$ so $I\nsubseteq\p$. This holds for any $\p\in D(f)$, so we conclude that $V(I)\cap D(f)=\O$. Therefore $f\in\sqrt[R]{I}$, thus there are $k\in\mathbb{N}$ and $\sum z^2$ in $A$ such that
\[
(f^{2k}+\sum z^2)\big{[}a(f^{2m}+\sum y^2)-b(f^{2n}+\sum x^2)\big{]}=0.
\]
Observing that $(f^{2k}+\sum z^2)\in\sum_{f}$,  we get $\frac{a}{f^{2n}+\sum x^2}=\frac{b}{f^{2m}+\sum y^2}$ in $\sum_{f}^{-1} A$. Hence $\psi$ is injective.

Let's show that $\psi$ is surjective. Let $s\in\OO_X(D(f))$, by Basic Lemma there are $g_1,\dots,g_n\in A$ such that $D(f)=D(g_1)\cup\cdots\cup D(g_n)$ and $s\upharpoonright D(g_i)=\frac{a_i}{g_i}$ for some $a_i\in A$. So for any $\p\in D(g_ig_j)=D(g_i)\cap D(g_j)$, there is $h_{\p}\notin\p$ such that $h_{\p}(a_ig_j-a_jg_i)=0$ in $A$. Let $I$ be the annihilator of $a_ig_j-a_jg_i$, so $h_{\p}\in I$, $h_{\p}\notin\p$ thus $\p\nsupseteq I$ and this holds for any $\p\in D(g_ig_j)$, therefore we have $V(I)\cap D(g_ig_j)=\O$ from which it follows that $g_ig_j\in\sqrt[R]{I}$. So there are $m_{ij}\in\mathbb{N}$ and $\sum^{ij}x^2$ in $A$ such that
\[
\big{[}(g_ig_j)^{2m_{ij}}+\sum\nolimits^{ij}x^2\big{]}(a_ig_j-a_jg_i)=0.
\]
since $A$ is a real ring there are $n_{ij}\in\mathbb{N}$ such that
\[
(g_ig_j)^{n_{ij}}(a_ig_j-a_jg_i)=0.
\]
So taking $m$ larger than all $n_{ij}$ we get $(g_ig_j)^m(a_ig_j-a_jg_i)=0$ for all $i,j\in\{1,\dots,n\}$. Rewriting as
\[
g_{j}^{m+1}(g_i^ma_i)-g_{i}^{m+1}(g_j^ma_j)=0
\]
and then replacing each $g_i$ by $g_i^{m+1}$, and $a_i$ by $g_i^ma_i$, we still have $s$ represented on on $D(g_i)$ by $\frac{a_i}{g_i}$ and furthermore, we have $g_ia_j=g_ja_i$ for all $i,j\in\{1,\dots,n\}$. On the other hand having $V((f))=\bigcap^{n}_{i=1}V((g_i))=V(\sum^{n}_{i=1}(g_i))$ we get $f\in\sqrt[R]{\sum^{n}_{i=1}(g_i)}$ and it follows that there are $b_1,\dots,b_k$ in $A$ such that $b_1g_1+\cdots+b_ng_n=f^{2k}+\sum x^2$ for some $k$. Now let $a=a_1b_1+\cdots+a_nb_n$. Then for each $j$ we have
\[
g_ja=\sum_{i}b_ia_ig_j=\sum_{j}b_ig_ia_j=(f^{2k}+\sum x^2)a_j,
\]
which means that $\frac{a}{f^{2k}+\sum x^2}=\frac{a_j}{g_j}$ on $D(g_j)$. So $\psi(\frac{a}{f^{2k}+\sum x^2})=s$ everywhere, which shows that $\psi$ is surjective, hence an isomorphism.
\end{proof}

{\bf Acknowledgment.} The research of the author was in part supported by a grant from IPM (No. 99030403).

\bibliography{reference}
\bibliographystyle{plain}
\end{document}